\documentclass[a4paper]{amsart}

\usepackage{amssymb}
\usepackage{amsthm}  
\usepackage{amsmath} 
\usepackage{amscd} 
\usepackage{float}
\usepackage[all]{xypic}
\usepackage{url}
\usepackage{hyperref}
\usepackage{color}
\usepackage{latexsym}
\usepackage{enumerate} 
\hypersetup{linktocpage}

\newcommand{\fg}{\mathfrak g}

\newcommand{\fs}{\mathfrak s}

\newcommand{\fn}{\mathfrak n}

\newcommand{\fm}{\mathfrak m}

\newcommand{\ad}{{\rm ad}}

\newcommand{\C}{\mathbb{C}}
\newcommand{\R}{\mathbb{R}}

\newcommand{\Imm}{\mathrm{Im}\,}

\theoremstyle{plain}
\newcounter{theo}
\newtheorem{thm}[theo]{Theorem}

\newtheorem{prop}[theo]{Proposition}
\newtheorem{lem}[theo]{Lemma}

\theoremstyle{definition}
\newtheorem{rem}[theo]{Remark}
\newtheorem{df}[theo]{Definition}
\newenvironment{pf}{{\noindent\bf Proof. }}{\hfill $\square$\medskip}

\begin{document}
\title[Construction of  counterexamples to the $2-$jet determination Chern-Moser Theorem in higher codimension]{Construction of  counterexamples to the $2-$jet determination Chern-Moser Theorem in higher codimension}
\author{Jan Gregorovi\v c and Francine Meylan}
\address{J. G., Faculty of Science, University of Hradec Kr\'alov\'e, Rokitansk\'eho 62, Hradec Kr\'alov\'e 50003, Czech Republic and Faculty of Mathematics, University of Vienna, Oskar Morgenstern Platz 1, 1090 Wien, Austria \newline
F. M., Department of Mathematics, University of Fribourg, CH 1700 Perolles, Fribourg}
 \email{jan.gregorovic@seznam.cz, francine.meylan@unifr.ch}
 \thanks{J. G. gratefully acknowledges support via Czech Science Foundation (project no. 19-14466Y) and partial support via FWF}

\maketitle

\begin{abstract}
We  first construct a counterexample of a  generic  quadratic  submanifold of  codimension $5$ in $\Bbb C^9$ which  admits a real analytic 
infinitesimal CR automorphism with homogeneous polynomial coefficients of degree $4.$ This example also resolves a question in the Tanaka prolongation theory that was open for more than 50 years.
 Then  we give  sufficient conditions to generate more  counterexamples to the  $2-$jet determination Chern-Moser Theorem in higher codimension. In particular, we construct examples of generic quadratic submanifolds with jet determination of arbitrarily high order.
\end{abstract}

\section{Introduction}
Let $M$ be a real-analytic  submanifold of $\C^N$ of codimension $d.$  Consider the set of germs of biholomorphisms $F$  at a point $p\in M$  such that $F(M)\subset M$. 
By the   work of Cartan \cite{Ca}, Tanaka \cite{Ta} and  Chern and Moser \cite{CM},  if the  codimension $d=1,$  every such $F$ is uniquely determined by its  first and second derivatives  at $p  $  provided that its  Levi map  at $p$ is non-degenerate. 


\begin{thm} \cite{CM}
Let $M$ be a real-analytic hypersurface through a point $p$ in $\C^N$ with non-degenerate Levi form at $p$. Let $F$, $G$ be two germs of biholomorphic maps preserving $M$. Then, if $F$ and $G$ have the same 2-jets at $p$, they coincide.
\end{thm}
Note that the result becomes false without any hypothesis on the Levi form (See for instance  \cite{bl-me1}). A generalization of this Theorem to  real-analytic submanifolds  $M$ of higher codimension $d>1$ has been proposed by Beloshapka in \cite{Be2} (and quoted many times by several authors)  under the hypothesis that  $M$ is  Levi generating (or equivalently of finite type with $2$ the only H\"ormander number) with  non-degenerate Levi map. Unfortunately, an error has been discovered  and explained in \cite{bl-me1}.

 In the first part of the paper, inspired by the technics developed in \cite{KMZ} and \cite{Bl-Me}, we  construct an example of a  generic (Levi generating  with non-degenerate Levi map)   quadratic  submanifold  that admits  an element in its stability group which has the same $2-$jet as the identity map but is not the identity map. In addition, this example is Levi non-degenerate in the sense of Tumanov. We  point out that if $M$ is strictly pseudoconvex, that is Levi non-degenerate  in the sense of Tumanov with a positivity condition,  then by a recent result of Tumanov\cite{tu3}, the  $2$-jet determination result holds in any codimension.
 
Moreover, this example solves an open question in the Tanaka prolongation theory that is related to  the weighted order of infinitesimal automorphisms. Roughly speaking, the question is to understand how large the positive depth of the Tanaka prolongation of a 
negatively graded nilpotent Lie algebra of  finite depth can be. (See \cite{DZ} for a precise statement).  For more than 50 years, only examples with positive weighted order less {or equal} in absolute value than minimum negative weighted order were known. Very recently, a preprint of Doubrov and Zelenko \cite{DZ}  has appeared with a counterexample. Let us emphasize that the example in the first part is also a counterexample to this open problem, and that  it was announced in the preprint of the second author \cite{Me} shortly before the counterexample of Doubrov and Zelenko.

In the second part of the paper, we use  different technics than in the first part. More precisely, we consider the Tanaka prolongation theory to obtain sufficient conditions to generate new examples, specifying the number of jets needed to determine  the given biholomorphisms. In particular, we obtain a  second counterexample of a generic  quadratic  submanifold of codimension $4$  in $\Bbb C^{10}$ which  admits a real analytic 
infinitesimal CR automorphism with homogeneous polynomial coefficients of degree $3.$ With the help of these two counterexamples, we construct a family of generic quadratic submanifolds  with jet determination of arbitrarily high order.  We point out that in codimension two,  $2$-jet determination holds \cite{Bl-Me}. The authors do not know if $2$-jet determination  also holds in codimension $3.$


We  also  mention that
finite jet determination problems for   submanifolds has attracted  much attention.
 We  refer  in particular  to the papers of Zaitsev \cite{Za}, Baouendi, Ebenfelt and Rothschild \cite{BER1}, Baouendi, Mir and Rothschild \cite{BMR}, Ebenfelt, Lamel and Zaitsev \cite{eb-la-za},  Lamel and Mir \cite{la-mi}, Juhlin \cite{ju}, Juhlin and Lamel \cite{ju-la}, Mir and Zaitsev \cite{mi-za} in the real analytic case, Ebenfelt \cite{eb}, Ebenfelt and Lamel \cite{eb-la}, Kim and Zaitsev \cite{ki-za}, Kolar, the  author and Zaitsev \cite{KMZ}  in the
$ \mathcal{C}^\infty$ case,  Bertrand and Blanc-Centi \cite{be-bl}, Bertrand, Blanc-Centi  and the second  author \cite {be-bl-me}, Bertrand and the second author \cite{be-me}, Tumanov \cite{tu3}  in the finitely smooth case.

The paper is organized as follows: In Section 2, we give the construction of  the counterexample, as it was described in \cite{Me}. In Section 3, we recall the necessary definitions and properties  needed in the sequel. In particular, we recall the definitions of non-degenerate Levi Tanaka Lie algebras and  their Tanaka prolongations. In Section 4, we state and prove the theorem generating counterexamples. (See Proposition \ref{int} and Theorem \ref{const}). In Section 5, we construct examples of generic quadratic submanifolds  with jet determination of arbitrarily high order.(See Theorem \ref{feriel1}).
\bigskip

\section{The  Example}
Let $M \subseteq \C^{9}$ be the  real submanifold of (real) codimension $5$ through $0$ given  in the coordinates $(z,w)=(z_1, \dots,z_4, w_1, \dots, w_5)  \in \C^{9},$ by 
\begin{equation}\label{Fe}
\begin{cases}
\Imm w_1= P_1 (z, \bar z)=z_1 \overline{ z_2} +  z_2 \overline{ z_1} \\
\Imm w_2=P_2 (z, \bar z)= -i z_1 \overline{ z_2} + i z_2 \overline{ z_1} \\
 \Imm w_3=P_3 (z, \bar z) = z_3 \overline{ z_2} + z_4 \overline{ z_1} + z_2 \overline{ z_3} + z_1 \overline{ z_4} \\
 \Imm w_4=P_4 (z, \bar z)= z_1 \overline{ z_1}  \\
 \Imm w_5=P_5 (z, \bar z) =z_2 \overline{ z_2}  \\
\end{cases}
\end{equation}

The matrices corresponding to the $P_i's$ are 
$$A_1=\left(\begin{array}{cccc}0&1&0&0\\1&0&0&0\\0&0&0&0\\0&0&0&0\end{array}\right)\qquad 
A_2=\left(\begin{array}{cccc}0&-i&0&0\\i&0&0&0\\0&0&0&0\\0&0&0&0\end{array}\right)\qquad 
A_3=\left(\begin{array}{cccc}0&0&0&1\\0&0&1&0\\0&1&0&0\\1&0&0&0\end{array}\right)\qquad $$
$$
A_4=\left(\begin{array}{cccc}1&0&0&0\\0&0&0&0\\0&0&0&0\\0&0&0&0\end{array}\right)\qquad 
A_5=\left(\begin{array}{cccc}0&0&0&0\\0&1&0&0\\0&0&0&0\\0&0&0&0\end{array}\right)\qquad 
$$

\begin{lem} The following holds:
\begin{enumerate}
\item the $A_i's$ are linearly independent, \\
\item the $A_i's$ satisfy the condition of Tumanov, that is, there is $c \in \Bbb R^d$ such that $\det{\sum c_jA_j}\ne 0.$
\end{enumerate}
\end{lem}
\begin{prop}
The real submanifold  $M$ given by \eqref{Fe} is Levi generating at $0,$ that is, of finite type with $2$ the only H\"ormander number, and its Levi map is  non-degenerate.
\end{prop}
\begin{proof}
This follows for instance from Proposition 8, Lemma 3 and Remark 4 in \cite{bl-me1}.
\end{proof}
\begin{rem} The following identity between  the $P_i's$ holds:
\begin{equation}\label{fer1}
{P_1}^2 + {P_2}^2 -4 P_4P_5=0.
\end{equation}
\end{rem}
The following  holomorphic vectors fields  are in $hol(M,0),$ the set of germs of real-analytic infinitesimal CR automorphisms at $0.$

\begin{enumerate}
\item $ X:=i(z_1
\dfrac{\partial}{\partial {z_3}}+z_2 \dfrac{\partial}{\partial {z_4}})$
\item $Y:= i(-iz_1\dfrac{\partial}{\partial {z_3}}+iz_2 \dfrac{\partial}{\partial {z_4}})$
\item $Z:= i(z_1\dfrac{\partial}{\partial {z_4}})$
\item $U:= i(z_2\dfrac{\partial}{\partial {z_3}})$
\end{enumerate}

\begin{lem} Let $P=(P_1, \dots, P_4).$
The following holds:
\begin{enumerate}
\item $X(P)=(0,0,iP_1,0,0)$
\item $Y(P)=(0,0,iP_2,0,0)$
\item $Z(P)=(0,0,iP_4,0,0)$
\item $U(P)=(0,0,iP_5,0,0).$
\end{enumerate}
\end{lem}

\begin{lem} The following identities  hold:

\begin{enumerate}
\item $P_1 (-Y(P)) +P_2 (X(P)=0$
\item $P_1X(P) + P_2(Y(P) + P_5 (-2Z(P)) + P_4 (-2U(P)) =0$
\item $P_2(-2Z(P)) + P_4(2Y(P) =0$
\item $P_2 (-2U(P)) + P_5 (2Y(P) =0$
\end{enumerate}
\end{lem}
With the help of the Lemmata, one obtains

\begin{thm}  
{The holomorphic vector field $T$ defined by
\begin{equation}
T= -\dfrac{1}{2}{w_1}^2Y + \dfrac{1}{2}{w_2}^2 Y +
w_1w_2 X -2 w_2w_5 Z -2w_2w_4 U +2 w_4w_5 Y
\end{equation} is in  $ hol (M,0).$}

Hence $2-$jet determination does not hold for germs of  biholomorphisms sending  $M$ to $M.$
\end{thm}
\begin{rem} {We will see in Section 4 that  the number of jets needed to characterize germs of  biholomorphisms sending  $M$ to $M$ is $4.$ }
\end{rem}
\begin{rem}
Notice  that the bound for the   number $k$ of jets  needed to determine uniquely any germ of  biholomorphism sending $M$ to $M$   is  $$k= (1+ \text{codim} \  M),$$  $M$ beeing a generic (Levi generating with non-degenerate Levi map) real-analytic submanifold: see Theorem 12.3.11, page 361 in \cite{BER}.   We  point out that  Zaitsev  obtained the bound  $k= 2(1+ \text{codim} \  M)$ in \cite{Za}.
\end{rem}
\bigskip

\section{Quadric models and Tanaka prolongation}

In this section, we will work with  quadric models  $M_0\subset\mathbb C^{n+k}$   that are $2$-degree polynomial submanifolds given by the following  system of equations
\begin{equation}\label{bg}
\Imm w_1=zH_1z^*,\cdots,\Imm w_k=zH_kz^*,\quad 
\end{equation} 
where $z\in\mathbb C^n,\,\,w\in\mathbb C^k,1\leq k\leq n^2$ and each $H_j$ is a $n\times n$ Hermitian matrix, i.e., $H_j^*=H_j$ holds for the conjugate transpose denoted by ${}^*.$

Recall that a (Levi) non--degenerate quadric   model in the sense of \cite{belold} is given by defining equations \eqref{bg} satisfying the following conditions:

\begin{enumerate}
\item the Hermitian matrices $H_j$ are linearly independent, and
\item the common kernel of all Hermitian matrices $H_j$ is trivial, i.e., $zH_jz^*=0$ for all $j$ implies $z=0$.
\end{enumerate}

The non--degenerate quadric models have many special properties.
  Firstly, they are weighted homogeneous for integral weights
$$[z_j]=1,\quad [w_j]=2.$$
 Further, they have a special infinitesimal CR automorphism 
$$E:=\sum_j z_j\frac{\partial}{\partial z_j}+2\sum_j w_j\frac{\partial}{\partial w_j}$$
called the Euler field or grading element. The Euler field (grading element) provides the structure of a graded Lie algebra to the Lie algebra of infinitesimal CR automorphisms of non--degenerate quadric models that is a decomposition $\fg=\fg_{a}\oplus\fg_{a+1}\dots \fg_{b-1}\oplus\fg_{b},$ compatible with Lie bracket in the sense
$$[\fg_{c},\fg_{d}]\subset \fg_{c+d}$$
(assuming $\fg_{e}=0$ for $e<a$ or $e>b$).
 Finally, the negative part $\fg_{-2}\oplus\fg_{-1}$ of the grading of Lie algebra of infinitesimal CR automorphism of non--degenerate quadric models is (infinitesimally) transitive, i.e., the real span of the real parts of these vector fields on $M_0$ is $TM_0$. In particular,
\begin{align*}
\fg_{-2}&=\{\sum q_j\frac{\partial}{\partial w_j}\}\\
\fg_{-1}&=\{\sum p_j\frac{\partial}{\partial z_j}+2i\sum_k zH_kp^*\frac{\partial}{\partial w_k}\},
\end{align*}
where $p\in\mathbb C^n,\,\,q\in\mathbb R^k$.
\begin{df}
A non--degenerate Levi Tanaka algebra (of a nondegenerate quadric model) is a graded Lie algebra $\fm=\fg_{-2}\oplus\fg_{-1}$ together with complex structure$J$ on $\fg_{-1}$ satisfying
\begin{enumerate}
\item $[\fg_{-1},\fg_{-1}]=\fg_{-2}$,
\item $[X,\fg_{-1}]=0,$    $X\in \fg_{-1}$, implies $X=0$ 
\item $[J(X),J(Y)]=[X,Y]$ for all $X,Y\in \fg_{-1}.$
\end{enumerate}
\end{df}
It is simple to check that the negative part $\fg_{-2}\oplus\fg_{-1}$ of the Lie algebra of infinitesimal CR automorphism with the Lie bracket taken with the opposite sign defines a non--degenerate Levi Tanaka algebra (at $z=0,w=0$) of the nondegenerate quadric model with $J$ induced by multiplication by $i$. We emphasize that the opposite sign is due to the difference of bracket of left and right invariant vector fields. In particular, the Lie bracket in the non--degenerate Levi Tanaka algebra of a nondegenerate quadric model given by defining equations \eqref{bg} is
$$[(q,p),(\tilde q,\tilde p)]=(2i(-pH_j\tilde p^*+\tilde pH_jp^*),0)$$
in the above coordinates $(q,p)$ of $\fg_{-2}\oplus\fg_{-1}$.

Conversely, for every non--degenerate Levi Tanaka algebra $\fm=\fg_{-2}\oplus\fg_{-1}$ we can reconstruct uniquely up to  holomorphic linear   change of coordinates a non--degenerate quadric model by formula
$$\Imm w:=\frac14[J(z),z],$$
see \cite[Lemma 3.3]{G18}.
Indeed, in the above situation $\Imm w_j=\frac12i(-izH_jz^*+zH_j(iz)^*)=zH_jz^*,$ because $J(z)=iz.$

Recall that  $\mathfrak{der}_0(\fm)$ is the space of the grading preserving derivations of $\fm$, that is,  the linear maps $f: \fg_{-i} \longrightarrow \fg_{-i}, \ i=1,2, $ such that $f([X,Y])=[f(X),Y]+[X,f(Y)]$.

Setting
\begin{equation}\label{ta1}\fg_0:=\{f\in \mathfrak{der}_0(\fm)|f(J(Y))=J(f(Y)) \rm{\ for\ all\ } Y\in \fg_{-1}\},
 \end{equation}
the infinitesimal CR automorphisms of the Levi nondegenerate quadric models with the Levi--Tanaka algebra $(\fm,J)$ can be computed by  the so called Tanaka prolongation from \cite{T}. Let us recall that Tanaka prolongation  of  $\fm \oplus \fg_0, $  where  $ \fg_0$  is defined by \ref{ta1},
 is the maximal nondegenerate graded Lie algebra  $\fg(\fm,\fg_0)$ containing the graded algebra $\fm \oplus \fg_0$, that is, 
\begin{enumerate}
\item  $\fg_i(\fm,\fg_0)=\fg_i,\ i \le 0.$
\item   If $X \in \fg_i$ with  $i>0,$ satisfies $[X,\fg_{-1} ]=0,$ then $X=0.$
\item $\fg(\fm,\fg_0)$ is the maximal graded Lie algebra satisfying (1) and (2).
\end{enumerate}

The result of \cite{T} is that for $i>0,$
\begin{gather*}
\fg_i=\{f\in \oplus_{j<0}\fg_j^*\otimes \fg_{j+i}:\ f([X,Y])=[f(X),Y]+[X,f(Y)] \\
\rm{\ for\ all\ } X,Y\in \fm\}
\end{gather*}
holds for the Tanaka prolongation $\fg$ of $\fm\oplus \fg_0.$  Let us emphasize that $f\in \fg_i, \ i>0,$ is uniquely determined by the component of $f$ in $\fg_{-1}^*\otimes \fg_{i-1}$, because $\fm$ is generated by $\fg_{-1}$. In \cite{T}, Tanaka proves the following result.

\begin{thm}\cite{T}
If $(\fm,J)$ is a non--degenerate Levi--Tanaka algebra, then $\fg_s=0$ for all $s$ large enough and the Tanaka prolongation $\fg$ of $(\fm,J)$ is a finite dimensional Lie algebra.
\end{thm}

Let us recall following \cite[Section 3.4]{G18}, how to construct the holomorphic vector fields corresponding to elements $X_b\in \fg_b$ of the Tanaka prolongation $\fg$ of the non--degenerate quadric model given by defining equations \eqref{bg}. For this we have to consider the adjoint representation $\ad$ given by the Lie bracket on the complexification $\fg_\C$ of $\fg$ and identify the abelian subalgebra $\fn_{-2}\oplus \fn_{-1}$ of $\fg_\C$ with our coordinates $(w,z)$, where $\fn_{-2}$ is the complexification of $\fg_{-2}$ and $\fn_{-1}$ is the $i$--eigenspace of $J$ in complexification of  $\fg_{-1}$. In this notation we get formula
\begin{equation} \label{eq1}
\begin{split}
\sum_{c+2d=b+1,\ c,d\geq 0} \frac{(-1)^{c+d}}{(c+d)!}(\ad(z)^c(\ad(w)^d(X_b)))_{\fn_{-1},j}\frac{\partial}{\partial z_j}\\
+\sum_{c+2d=b+2,\ c,d\geq 0} \frac{(-1)^{c+d}}{(c+d)!}(\ad(z)^c(\ad(w)^d(X_b)))_{\fn_{-2},j}\frac{\partial}{\partial w_j}
\end{split}
\end{equation}
the holomorphic vector fields corresponding to elements $X_b\in \fg_b$,
where ${}_{\fn_{-i},j}$ means the projection from $\fg_\C$ to jth--component of $\fn_{-i}$ (along the $-i$--eigenspace of $J$ in complexification of  $\fg_{-1}$). Indeed, since $c+2d=b+i$, we project elements of complexification of $\fg_{-i}$.

We denote by $R$ the radical of the Tanaka prolongation $\fg$ of $(\fm,J)$. Let us recall that the Levi decomposition Theorem (see for instance\cite{FH}) ensures that the semisimple Lie algebra $\fg/R$ is isomorphic to a  (not necessarily unique) subalgebra $\fs\subset \fg$, i.e., $\fg=\fs\oplus_\rho R$, where $\rho: \fs\to \frak{gl}(R)$ is the representation induced by the Lie bracket $[\fs,R]\subset R$.    Medori ad Nacinovich show in \cite[Theorem 3.27]{M}  that we can choose $\fs\subset \fg$ such that $[E,\fs]\subset \fs$, where $E$ is the Euler field compatible with the complex structure  $J$ on $\fg_{-1}$.  More precisely, they show that 
\begin{align*}
\fs&=\fs_{-2}\oplus\fs_{-1}\oplus\fs_0\oplus\fs_1\oplus\fs_2 \\
R&=R_{-2}\oplus \dots \oplus R_{b},
\end{align*}
with $J(\fs_{-1})\subset \fs_{-1}$ and $ J(R_{-1})\subset R_{-1}.$
We emphasize
 that $\{-2, \dots, b\}$ are the weights of the vectors fields.

\section{The construction of Examples}

Let us return to Levi decomposition $\fg=\fs\oplus_\rho R$ and describe, how it is reflected in the defining equations  \eqref{bg} (in the compatible coordinates). Firstly,  $\dim_{\R}(\fs_{-2}\oplus R_{-2})>0$ is the codimension and $\dim_{\C}(\fs_{-1}\oplus R_{-1})>0$ is the complex dimension of (Levi) non--degenerate quadric model and we decompose the coordinates to coordinates $w$ corresponding to $\fs_{-2}$,  $w'$ corresponding to $R_{-2}$, $z$ corresponding to $\fs_{-1}$ and  $z'$ corresponding to $R_{-1}$. Since  $[\fs_{-1},R_{-1}]\subset R_{-2}$, we see that 
$$\Imm w'_j=\mbox{Re}(zP_j(z')^*)+z'Q_j(z')^*,$$
where $zP(z')^*:=-2\rho(z)(z')$ and $Q_j$ are a Hermitian matrices completely determined by the bracket $[R_{-1},R_{-1}]\subset R_{-2}.$

Further, we see that if  $\dim_{\R}(\fs_{-2})>0,$ then the equations $$\Imm w_j=zH_jz^*$$
are completely determined by the bracket $[\fs_{-1},\fs_{-1}]\subset \fs_{-2}$  which  does not depend on  $z'$. In particular, we can decompose the equations according to simple factors $\fs^i$ of  $\fs$. For each simple factor $\fs^i=\fs_{-2}^i\oplus\fs_{-1}^i\oplus\fs_0^i\oplus\fs_1^i\oplus\fs_2^i$ with $\dim_{\R}(\fs_{-2}^i)>0$ the equations $\Imm w_{j_i}=zH_{j_i}z^*$ define a real submanifold $M_{s^i}$ in complex space with complex dimension $\dim_{\C}(\fs^i_{-1})>0$ and codimension $\dim_{\R}(\fs^i_{-2})>0$ with Lie algebra of infinitesimal CR automorphisms $\fs^i$. Since real submanifolds $M_{s^i}$ are classified in \cite{MS,A}, we can consider them as the starting point of our investigation. If there are simple factors $\fs^i$ of  $\fs$ with $\dim_{\R}(\fs_{-2}^i)=0$ and $\dim_{\C}(\fs^i_{-1})>0$, then the corresponding $z$ variables do not appear in the equations $\Imm w_{j_i}=zH_{j_i}z^*$, but only in equations $\Imm w'_j=\mbox{Re}(zP_j(z')^*)+z'Q_j(z')^*.$

Now, in addition to Euler field (grading element) $E,$ there are elements $E_{s^i}\in \fs^i$ providing the grading on $\fs^i$.  We set $$E_s: =\sum_i E_{s^i} .$$ Since $\rho(E_{s})$ acts diagonalizably (as an element of Cartan subalgebra) on $V$, we can decompose $V$ according to its eigenvalues. We obtain the following result that allow us to estimate the jet determination in terms of the eigenvalues of $\rho(E_{s})$.

\begin{prop}\label{int}
Let $\fg=\fs\oplus_\rho R$ be the Levi decomposition compatible with grading of the Lie algebra of infinitesimal CR automorphisms of a nondegenerate quadric model. Suppose $E_s\in \fs$ is the element providing the grading on $\fs$. Suppose $W$ is an irreducible subrepresentation of $\rho$ in $R$ decomposing as $W=W_{-2}\oplus \dots \oplus W_{c}$ w.r.t. eigenvalues of the Euler vector field $E$. Then:
\begin{enumerate}
\item the eigenvalue $K_{max}$ of highest weight vector is the largest eigenvalue of $\rho(E_s)$ on $W$ and the eigenvalue $K_{min}$ of lowest weigh vector is the smallest eigenvalue of $\rho(E_s)$ on $W$.
\item $W_i$ is the $i+K_{min}+2$ eigenspace of $\rho(E_s)$ in $W$ and $c=K_{max}-K_{min}-2$.
\item Infinitesimal CR automorphisms in $W_{c}$ have weighted degree $K_{max}-K_{min}-2$ and are at least $K$--jet determined, where $K$ is $\frac{K_{max}-K_{min}}{2}$ rounded down.
\end{enumerate}
\end{prop}
\begin{pf}
Since $W$ is irreducible, it is generated from single highest weight vector by actions of elements of $\rho(\fs_{-2}\oplus\fs_{-1}\oplus\fs_0)$ or from single lowest weight vector by actions of elements of $\rho(\fs_{2}\oplus\fs_{1}\oplus\fs_0)$. Therefore, the grading $W=W_{-2}\oplus \dots \oplus W_{c}$ is completely determined by eigenvalues  of  $\rho(E_s)$ on $W$ up to a shift. If $K_{max}$ is the largest eigenvalue of  $\rho(E_s)$ on $W$ and  $K_{min}$ is the smallest eigenvalue of  $\rho(E_s)$ on $W$, then clearly $W_{-2}$ is the $K_{min}$ eigenspace of  $\rho(E_s)$ and the shift is $K_{min}+2$. Therefore $c+K_{min}+2=K_{max}$ and the first two claims follow. The second claim is consequence of \cite[Corollary 3.6]{G18}.
\end{pf}

\begin{rem}
The proposition suggests, where to start:

\begin{enumerate}
\item We have to restrict our selves to the  irreducible subrepresentation $$V=V_{-2}\oplus \dots \oplus V_{c}$$ of $R$ with largest $$c=K_{max}-K_{min}-2,$$ because only here we have an a priory estimate on jet determinacy in the eigenvalues of $\rho(E_s)$.

\item Since $c$ does not depend on the Lie bracket on $V$, we can assume that $$[V,V]=0$$ without changing jet determination of the example.

\item Since we want to start with the known examples  $M_s$ from \cite{MS,A}, we assume that $$\fs=\fs_{-2}\oplus\fs_{-1}\oplus\fs_0\oplus\fs_1\oplus\fs_2$$  is simple and $\dim_{\R}(\fs_{-2})>0$. So we start with defining equations $$\Imm w_j=zH_jz^*$$ for $M_s$ and Euler field $E_s$.

\item So we want to add equations
$$\Imm w'_j=\mbox{Re}(zP_j(z')^*),$$
corresponding to an irreducible $\mathbb{K}$--representation (real, complex or quaternionic representation for $\mathbb{K}=\mathbb{R},\C$ or $\mathbb{H}$) $\rho: \fs\to \frak{gl}(V,\mathbb{K})$ such that $$K_{max}-K_{min}>5$$ holds for the minimal/maximal eigenvalues of $\rho(E_s)$ on $V$.

\item We can define abstractly the Euler field (grading element) $E$ as $E_s$ up to shift by $K_{min}+2$ on $V$ and define $$\fg_{-2}:=\fs_{-2}\oplus V_{-2},\ \fg_{-1}:=\fs_{-1}\oplus V_{-1}$$ w.r.t. to eigenvalues of $E$.
\end{enumerate}
\end{rem}
The assumptions in (1)--(5) for general representation $\rho$ do not ensure that we get an example with high jet determinacy, because there does not have to be complex structure $J$ on $V_{-1}$ that would make $\fg_{-2}\oplus\fg_{-1}$ into Levi Tanaka algebra, i.e., $zP(z')^*:=-2\rho(z)(z')$ is not well--defined without the conjugation provided by the complex structure. Further, even if $(\fg_{-2}\oplus\fg_{-1},J)$ is Levi--Tanaka algebra then $\fs\oplus_\rho V$ does not have to be contained in Lie algebra of infinitesimal CR automorphisms of the corresponding quadric model given by the above equations. The following Theorem provides sufficient conditions for resolving these problems.

\begin{thm}\label{const}
Let $\fs, \rho,K_{max},K_{min}, V=V_{-2}\oplus \dots \oplus V_{K_{max}-K_{min}-2}$ satisfy the assumptions as in the above (1)--(5). If
\begin{enumerate}
\item $V_{-1}$ is a complex representation of $\fs_0$ and the corresponding complex structure $J$ on $\fg_{-1}$ is satisfying $$\rho(J(X)(J(Y)=\rho(X)(Y)$$ for all $X\in \fs_{-1},Y\in V_{-1},$
\item $V_{0}$ acts complex linearly as a map from $\fs_{-1}$ to $V_{-1}$,
\end{enumerate}
then $(\fg_{-2}\oplus\fg_{-1},J)$ is a non--degenerate Levi Tanaka algebra and the nondegenerate quadric model $M$ defined as above has Lie algebra  $\fs\oplus_\rho V\oplus \mathbb{K}$ of infinitesimal CR automorphisms, where $\mathbb{K}$ acts as $0$ of $\fs$ and by scalar multiplication on $V$. In particular, infinitesimal CR automorphisms in $V_{K_{max}-K_{min}-2}$ have weighted degree $K_{max}-K_{min}-2$ and are at least $K$--jet determined, where $K$ is $\frac{K_{max}-K_{min}}{2}$ rounded down.
\end{thm}
\begin{pf}
The condition (1) implies that $(\fg_{-2}\oplus\fg_{-1},J)$ is a non--degenerate Levi Tanaka algebra. The condition (2) implies that $\fs\oplus_\rho V$ is contained in the Tanaka prolongation of the Levi--Tanaka algebra $(\fg_{-2}\oplus\fg_{-1},J)$. Then from the Schur's Lemma follows, that $\fs\oplus_\rho V\oplus \mathbb{K}$ is the Tanaka prolongation of the Levi--Tanaka algebra $(\fg_{-2}\oplus\fg_{-1}, J)$.

The construction in \cite[Lemma 3.3]{G18} provides the above realization of the Levi--Tanaka algebra $(\fg_{-2}\oplus\fg_{-1},J)$ as a nondegenerate quadric model with Lie algebra  $\fs\oplus_\rho V\oplus \mathbb{K}$ of infinitesimal CR automorphisms. Proposition above provides the claim on jet determinacy.
\end{pf}

The examples in this paper were obtained using this Theorem \ref{const}. The codimension 4 example was obtained as follows:

We considered the codimension $3$ submanifold $M_0$ in $\C^6$ given by the following defining equations:
\begin{align*}
\Imm w_1&=-iz_1\bar{z}_2+iz_2\bar{z}_1\\
\Imm w_2&=-iz_2\bar{z}_3+iz_3\bar{z}_2\\
\Imm w_3&=-iz_1\bar{z}_3+iz_3\bar{z}_1
\end{align*}
that has $|2|$--graded Lie algebra $\fs=\frak{so}(3,5)$ of infinitesimal CR automorphisms. Let us check that the real irreducible representation $V$ of $\frak{so}(3,5)$ with the highest weight $\lambda_3+\lambda_4$ (where $\lambda_i$ is the $i$--th fundamental weight) satisfies all the conditions (1)--(2) of Theorem \ref{const}. Firstly,  $K_{max}=3, K_{min}=-3$, $V=V_{-2}\oplus\dots \oplus V_{4}$ and $V_{-1}$ is a standard complex representation of $\fg_0=\frak{sl}(3,\mathbb{R})\oplus \mathbb{C}$. If we look on the weights of the complexification of $\fs\oplus_\rho V$, then we observe that the additional equation corresponding to $\rho(z)(z')$ is 
$$\Imm w'_1=z_1\bar{z}_1'+z_1'\bar{z}_1+z_2\bar{z}'_2+z_2'\bar{z}_2+z_3\bar{z}'_3+z_3'\bar{z}_3,$$ where $z_1,z_2,z_3\in \fg_{-1}\otimes \mathbb{C}$ and $z_1',z_2',z_3'\in V_{-1}\otimes \mathbb{C}$. Therefore  the condition (1) is satisfied. Similarly, we see from the weight spaces in $V_{0}\otimes \mathbb{C}$ that $V_{0}$ acts complex linearly as a map from $\fs_{-1}$ to $V_{-1}$.

Therefore, all the conditions of Theorem \ref{const} are fulfilled and we obtain that the following submanifold in $\C^{10}$ given by the following defining  equations
\begin{align*}
\Imm w_1&=-iz_1\bar{z}_2+iz_2\bar{z}_1\\
\Imm w_2&=-iz_2\bar{z}_3+iz_3\bar{z}_2\\
\Imm w_3&=-iz_1\bar{z}_3+iz_3\bar{z}_1\\
\Imm w'_1&=z_1\bar{z}_1'+z_1'\bar{z}_1+z_2\bar{z}'_2+z_2'\bar{z}_2+z_3\bar{z}'_3+z_3'\bar{z}_3
\end{align*}
has infinitesimal CR automorphism in $V_4$ that has weighted order $4$ and is $3$--jet determined. The 
formula \eqref{eq1} provides the following formula for the corresponding holomorphic vector field:
\begin{align*}
&-w_1 w_3(iz_3\dfrac{\partial}{\partial {z_1'}}+iz_1\dfrac{\partial}{\partial {z_3'}})
+w_2 w_3(iz_2\dfrac{\partial}{\partial {z_1'}}+iz_1\dfrac{\partial}{\partial {z_2'}})
-w_1 w_2(iz_3\dfrac{\partial}{\partial {z_2'}}+iz_2\dfrac{\partial}{\partial {z_3'}})\\
&
+iw_3^2(z_1\dfrac{\partial}{\partial {z_1'}})+iw_2^2(z_2\dfrac{\partial}{\partial {z_2'}})++iw_1^2(z_3\dfrac{\partial}{\partial {z_3'}}),
\end{align*}
where the vector fields in the braces are rigid holomorphic vector fields that are elements of $V_0$.

The codimension 5 example was obtained as follows:

We considered the codimension $4$ submanifold $M_0$ in $\C^6$ given by the following defining equations:
\begin{align*}
\Imm w_1&=z_1\bar{z}_2+z_2\bar{z}_1\\
\Imm w_2&=-iz_1\bar{z}_2+iz_2\bar{z}_1\\
\Imm w_3&=z_1\bar{z}_1\\
\Imm w_4&=z_2\bar{z}_2\\
\end{align*}
that has $|2|$--graded Lie algebra $\fs=\frak{su}(2,3)$ of infinitesimal CR automorphisms. Let us check that the real irreducible representation $V$ of $\frak{su}(2,3)$ with the highest weight $\lambda_2+\lambda_3$ (where $\lambda_i$ is the $i$--th fundamental weight) satisfies all the conditions (1)--(2) of Theorem \ref{const}. Firstly, $K_{max}=4, K_{min}=-4$,   $V=V_{-2}\oplus\dots \oplus V_{6}$ and $V_{-1}$ is a standard complex representation of $\fg_0=\frak{gl}(2,\mathbb{C})$. If we look on the weights of the complexification of $\fs\oplus_\rho V$, then we observe that the additional equation corresponding to $\rho(z)(z')$ is 
$$\Imm w'_1=z_1\bar{z}'_2+z_2'\bar{z}_1+z_2\bar{z}'_1+z_1'\bar{z}_2,$$ where $z_1,z_2\in \fg_{-1}\otimes \mathbb{C}$ and $z_1',z_2'\in V_{-1}\otimes \mathbb{C}$. Therefore  the condition (1) is satisfied. Similarly, we see from the weight spaces in $V_{0}\otimes \mathbb{C}$ that $V_{0}$ acts complex linearly as a map from $\fs_{-1}$ to $V_{-1}$.

Therefore, all the conditions of Theorem \ref{const} are fulfilled, and we obtain that the following submanifold in $\C^9,$ that was given in the first part of the paper,
\begin{align*}
\Imm w_1&=z_1\bar{z}_2+z_2\bar{z}_1\\
\Imm w_2&=-iz_1\bar{z}_2+iz_2\bar{z}_1\\
\Imm w_3&=z_1\bar{z}_1\\
\Imm w_4&=z_2\bar{z}_2\\
\Imm w'_1&=z_1\bar{z}'_2+z_2'\bar{z}_1+z_2\bar{z}'_1+z_1'\bar{z}_2
\end{align*}
has infinitesimal CR automorphism in $V_6$ that has weighted order $6$ and is $4$--jet determined. The formula \eqref{eq1} provides the following formula for the corresponding holomorphic vector field:
\begin{align*}
& (-3w_1^4-6w_1^2w_2^2+24w_1^2w_3w_4-3w_2^4+24w_2^2w_3w_4-48w_3^2w_4^2 ) \dfrac{\partial}{\partial {w_1'}}\\
&+( 2w_1^3+2w_1w_2^2-8w_1w_3w_4)(2w_1\dfrac{\partial}{\partial {w_1'}}+z_1\dfrac{\partial}{\partial {z'_1}}+z_2 \dfrac{\partial}{\partial {z'_2}})\\
&+(2w_1^2w_2+2w_2^3-8w_2w_3w_4)(2w_2\dfrac{\partial}{\partial {w_1'}}-iz_1\dfrac{\partial}{\partial {z'_1}}+iz_2 \dfrac{\partial}{\partial {z'_2}})\\
&+(-4w_1^2w_3-4w_2^2w_3+16w_3^2w_4)(2w_4\dfrac{\partial}{\partial {w_1'}}+z_2\dfrac{\partial}{\partial {z'_1}})\\
&+(-4w_1^2w_4-4w_2^2w_4+16w_3w_4^2)(2w_3\dfrac{\partial}{\partial {w_1'}}+z_1\dfrac{\partial}{\partial {z'_2}}),
 \end{align*}
where the vector fields in the braces are non rigid holomorphic vector fields that are elements of $V_0$.

It is simple to produce counterexamples of higher codimension $k>5$ by adding further equations $$w_{4+j}=z_{2+j}\bar{z}_{2+j}$$
for $j=1,\dots,k-5$ with Lie algebra of infinitesimal CR automorphisms $(\frak{su}(2,3)\oplus_{\rho} V\oplus \mathbb{R})\oplus \bigoplus_j \frak{su}(1,2)$.
We then obtain the following Theorem
\begin{thm}\label{feriel}
For any codimension $k>5$, there is a generic quadratic submanifold $M$ in $\C^{2k-1}$ of codimension $k$ such that $4-$jets are  required (and not less) to determine uniquely  germs of  biholomorphisms sending  $M$ to $M.$
\end{thm}
\begin{rem}
The CR dimension can clearly be arbitrarily enlarged by taking quadric in larger space, however, the jet determination of the examples remains still $4$.
\end{rem}
\begin{rem}
The question of  $2$-jet determination in codimension $3$ remains still open.
\end{rem}

\section{Examples of jet determination of arbitrarily high order}

In this section, we use the Theorem \ref{const} to show that for any  $n,$ there are examples of codimension depending quadraticaly on $n,$ that have  $n$-jet determination of infinitesimal CR automorphisms. The examples extend to higher rank  what do you mean by rank.

We considered the codimension $\frac{(n-1)n}{2}$ submanifold $M_0$ in $\C^{\frac{n(n+1)}{2}}$ given by the following defining equations:
\begin{align*}
\Imm w_1&=-iz_1\bar{z}_2+iz_2\bar{z}_1\\
&\vdots\\
\Imm w_{n-1}&=-iz_{n-1}\bar{z}_n+iz_n\bar{z}_{n-1}\\
\Imm w_n&=-iz_1\bar{z}_3+iz_3\bar{z}_1\\
&\vdots\\
\Imm w_{\frac{(n-1)n}{2}}&=-iz_{1}\bar{z}_{n}+iz_{n}\bar{z}_{1}
\end{align*}
that has $|2|$--graded Lie algebra $\fs=\frak{so}(n,n+2)$ of infinitesimal CR automorphisms. The example in  the previous section is a special case for $n=3$. Let us check that the real irreducible representation $V$ of $\frak{so}(3,5)$ with the highest weight $\lambda_{n}+\lambda_{n+1}$ (where $\lambda_i$ is the $i$--th fundamental weight) satisfies all the conditions (1)--(2) of Theorem \ref{const}. Firstly,  $K_{max}=n, K_{min}=-n$, $V=V_{-2}\oplus\dots \oplus V_{2n-2}$ and $V_{-1}$ is a standard complex representation of $\fg_0=\frak{sl}(n,\mathbb{R})\oplus \mathbb{C}$. If we look on the weights of the complexification of $\fs\oplus_\rho V$, then we observe that the additional equation corresponding to $\rho(z)(z')$ is 
$$\Imm w'_1=\sum_{j=1}^n z_j\bar{z}'_j+z_j'\bar{z}_j,$$ where $z_1,\dots,z_n\in \fg_{-1}\otimes \mathbb{C}$ and $z_1',\dots,z_n'\in V_{-1}\otimes \mathbb{C}$. Therefore  the condition (1) is satisfied. Similarly, we see from the weight spaces in $V_{0}\otimes \mathbb{C}$ that $V_{0}$ acts complex linearly as a map from $\fs_{-1}$ to $V_{-1}$.

Therefore, all the conditions of Theorem \ref{const} are fulfilled and we obtain that the following submanifold in $\C^{\frac{(n+2)(n+1)}{2}}$ given by the following defining  equations
\begin{align*}
\Imm w_1&=-iz_1\bar{z}_2+iz_2\bar{z}_1\\
&\vdots\\
\Imm w_{n-1}&=-iz_{n-1}\bar{z}_n+iz_n\bar{z}_{n-1}\\
\Imm w_n&=-iz_1\bar{z}_3+iz_3\bar{z}_1\\
&\vdots\\
\Imm w_{\frac{(n-1)n}{2}}&=-iz_{1}\bar{z}_{n}+iz_{n}\bar{z}_{1}\\
\Imm w'_1&=\sum_{j=1}^n z_j\bar{z}'_j+z_j'\bar{z}_j
\end{align*}
has infinitesimal CR automorphism in $V_{2n-2}$ that has weighted order $2n-2$ and is $n$--jet determined. The formula \eqref{eq1} is now too long to provide a reasonable formula for the infinitesimal CR automorphism in $V_{2n-2}$.

The second example extends to higher ranks as follows for even $n=2m$:

We considered the codimension $m^2$ submanifold $M_0$ in $\C^{m+m^2}$ given by the following defining equations:
\begin{align*}
\Imm w_1&=z_1\bar{z}_2+z_2\bar{z}_1\\
&\vdots\\
\Imm w_{m-1}&=z_{m-1}\bar{z}_m+z_m\bar{z}_{m-1}\\
\Imm w_n&=z_1\bar{z}_3+z_3\bar{z}_1\\
&\vdots\\
\Imm w_{\frac{(m-1)m}{2}}&=z_{1}\bar{z}_{m}+z_{m}\bar{z}_{1}\\
\Imm w_{\frac{(m-1)m}{2}+1}&=-iz_1\bar{z}_2+iz_2\bar{z}_1\\
&\vdots\\
\Imm w_{\frac{(m-1)m}{2}+m-1}&=-iz_{m-1}\bar{z}_m+iz_m\bar{z}_{m-1}\\
\Imm w_{\frac{(m-1)m}{2}++m}&=-iz_1\bar{z}_3+iz_3\bar{z}_1\\
&\vdots\\
\Imm w_{(m-1)m}&=-iz_{1}\bar{z}_{m}+iz_{m}\bar{z}_{1}\\
\Imm w_{(m-1)m+1}&=z_1\bar{z}_1\\
&\vdots\\
\Imm w_{m^2}&=z_m\bar{z}_m
\end{align*}
that has $|2|$--graded Lie algebra $\fs=\frak{su}(m,m+1)$ of infinitesimal CR automorphisms. The example in previous section is special case for $m=2$. Let us check that the real irreducible representation $V$ of $\frak{su}(m,m+1)$ with the highest weight $\lambda_m+\lambda_{m+1}$ (where $\lambda_i$ is the $i$--th fundamental weight) satisfies all the conditions (1)--(2) of Theorem \ref{const}. Firstly, $K_{max}=n, K_{min}=-n$,   $V=V_{-2}\oplus\dots \oplus V_{2n-2}$ and $V_{-1}$ is a standard complex representation of $\fg_0=\frak{gl}(m,\mathbb{C})$. If we look on the weights of the complexification of $\fs\oplus_\rho V$, then we observe that the additional equation corresponding to $\rho(z)(z')$ is 
$$\Imm w'_1=\sum_{j=1}{m}z_j\bar{z}'_{m+1-j}+z_{m+1-j}'\bar{z}_j,$$ where $z_1,\dots,z_m\in \fg_{-1}\otimes \mathbb{C}$ and $z_1',\dots,z_m'\in V_{-1}\otimes \mathbb{C}$. Therefore  the condition (1) is satisfied. Similarly, we see from the weight spaces in $V_{0}\otimes \mathbb{C}$ that $V_{0}$ acts complex linearly as a map from $\fs_{-1}$ to $V_{-1}$.

Therefore, all the conditions of Theorem \ref{const} are fulfilled, and we obtain that the following submanifold in $\C^{(m+1)^2},$ that was given in the first part of the paper,
\begin{align*}
\Imm w_1&=z_1\bar{z}_2+z_2\bar{z}_1\\
&\vdots\\
\Imm w_{m-1}&=z_{m-1}\bar{z}_m+z_m\bar{z}_{m-1}\\
\Imm w_n&=z_1\bar{z}_3+z_3\bar{z}_1\\
&\vdots\\
\Imm w_{\frac{(m-1)m}{2}}&=z_{1}\bar{z}_{m}+z_{m}\bar{z}_{1}\\
\Imm w_{\frac{(m-1)m}{2}+1}&=-iz_1\bar{z}_2+iz_2\bar{z}_1\\
&\vdots\\
\Imm w_{\frac{(m-1)m}{2}+m-1}&=-iz_{m-1}\bar{z}_m+iz_m\bar{z}_{m-1}\\
\Imm w_{\frac{(m-1)m}{2}+m}&=-iz_1\bar{z}_3+iz_3\bar{z}_1\\
&\vdots\\
\Imm w_{(m-1)m}&=-iz_{1}\bar{z}_{m}+iz_{m}\bar{z}_{1}\\
\Imm w_{(m-1)m+1}&=z_1\bar{z}_1\\
&\vdots\\
\Imm w_{m^2}&=z_m\bar{z}_m\\
\Imm w'_1&=\sum_{j=1}{m}z_j\bar{z}'_{m+1-j}+z_{m+1-j}'\bar{z}_j
\end{align*}
has infinitesimal CR automorphism in $V_{2n-2}$ that has weighted order $2n-2$ and is $n$--jet determined. The formula \eqref{eq1} is again too long to provide a reasonable formula for the infinitesimal CR automorphism in $V_{2n-2}$.

With the possible addition of further quadrics in new variables, we then obtain the following  theorem  for  jet determination of arbitrarily high order, which can be seen as   an analogous result  to Theorem \ref{feriel}.
 
 \begin{thm}\label{feriel1}
For any even $n=2m$ and any $k>m^2$, there is a generic quadratic submanifold $M$ in $\C^{2k-m^2+2m-1}$ of codimension $k$ such that $n-$jets are  required (and not less) to determine uniquely  germs of  biholomorphisms sending  $M$ to $M.$

For any odd $n$ and any $k>\frac{(n-1)n}{2}$, there is a generic quadratic submanifold $M$ in $\C^{2k-\frac12n^2+\frac{5}{2}n-1}$ of codimension $k$ such that $n-$jets are  required (and not less) to determine uniquely  germs of  biholomorphisms sending  $M$ to $M.$
\end{thm}
\begin{rem}
The CR dimension can clearly be arbitrarily enlarged by taking quadric in larger space, however, the jet determination of the examples remains still $n$.
\end{rem}

\end{document}